\documentclass[12pt,onecolumn,twoside]{JCTA}

\usepackage{multicol}

\usepackage{amsmath}
\usepackage{amsthm}
\usepackage{amsfonts}
\usepackage{amssymb}
\usepackage{mathrsfs}

\usepackage{lettrine}

\usepackage[sort]{cite}

\usepackage{times}

\usepackage{booktabs}

\usepackage{wasysym}

\usepackage{fancyhdr}

\usepackage{ulem}

\allowdisplaybreaks

\newtheoremstyle{mystyle}{1pt}{1pt}{\normalfont}{\parindent}{\bfseries}{}{1em}{}
\theoremstyle{mystyle}
\newtheorem{Thm}{Theorem}
\newtheorem{Asu}{Assumption}
\newtheorem{Pro}{Proposition}
\newtheorem{Def}{Definition}
\newtheorem{Lem}{Lemma}

\newtheorem{Rem}{Remark}

\newtheorem{Step}{Step}

\makeatletter
\renewenvironment{proof}[1][\proofname]{\par
  \pushQED{\qed}%
  \normalfont 
  \trivlist
  \item[\hskip 1.6em
        \bfseries
    #1]\hskip .5em\ignorespaces
}{%
  \endtrivlist\@endpefalse
} \makeatother

\usepackage{amsfonts}
\usepackage{amsmath}
\usepackage{amsthm}
\usepackage{amssymb}

\newcommand{\br}{\mathbb R}

\newcommand{\tow}{\rightharpoonup}
\numberwithin{equation}{section}
\newtheorem{theorem}{Theorem}[section]
\newtheorem{lemma}[theorem]{Lemma}
\newtheorem{corollary}[theorem]{Corollary}

\newtheorem{problem}{Problem}[section]

\begin{document}                                 
\makeatletter
\def\@autr{{M.Boukrouche and D.A.Tarzia}}             
\makeatother

\begin{frontmatter}     

\title{Convergence of optimal control problems governed by second kind parabolic variational inequalities}

\author[1]{ Mahdi BOUKROUCHE},{}   
\author[2]{ Domingo A. TARZIA},{}

\address[1] {Lyon University, UJM F-42023, CNRS UMR 5208, Institut
Camille Jordan, 23 Paul Michelon,  42023 Saint-Etienne Cedex 2, France.
  E-mail: Mahdi.Boukrouche@univ-st-etienne.fr; }
\address[2] {Departamento de Matem\'atica-CONICET, FCE,
Univ. Austral, Paraguay 1950, S2000FZF Rosario, Argentina.
E-mail: DTarzia@austral.edu.ar}                   


\begin{abstract}
We consider a family of optimal control problems where the control variable is given
by a boundary condition of Neumann type. This family is governed by  parabolic variational
inequalities of the second kind.
We prove the strong  convergence of the  optimal controls and state systems associated to this family
to a similar  optimal control problem. This work solves the open problem left by the authors in IFIP TC7 CSMO2011.
\end{abstract}

\end{frontmatter}

\makeatletter
\renewcommand{\headrulewidth}{0pt}
\fancypagestyle{plain}{\fancyhf{} \fancyhead[LE,LO]{\small\it J
Control Theory Appl} \lfoot{} \cfoot{} \rfoot{}}
\thispagestyle{plain}

\begin{multicols}{2}      

\section{Introduction}
The motivation of this paper is to prove the strong  convergence of the optimal controls (borders)
and state systems  associated to a family of  second kind parabolic variational inequalities.
With this paper, we solve the open question, left  in  \cite{MB-DT4} and
we  generalize our work \cite{MB-DT3}, to study the {\it Control border}.

To illustrate the problem considered, we consider in the following,
just  as examples, two  free boundary problems which leads to
second kind parabolic variational inequalities.

We assume  that the boundary of a multidimensional
regular domain $\Omega$ is given by 
$\partial\Omega=\Gamma_{1}\cup\Gamma_{2}\cup\Gamma_{3}$ with
$meas(\Gamma_{1}) >0$ and $meas(\Gamma_{3}) >0$.
We consider a family of optimal control problems where the control variable is given
by a boundary condition of Neumann type whose state system is
governed by a free boundary problem with Tresca conditions on a portion
$\Gamma_{2}$ of the boundary,  with a flux $f$ on $\Gamma_{3}$ {\it as
the control variable}, given by:
\begin{problem}\label{pb4.1}
\begin{eqnarray*}
 \dot{u} - \Delta u = g  \quad \mbox{in} \quad \Omega\times(0, T), \label{III.1}
\end{eqnarray*}
\begin{eqnarray*}
&&\left|{\partial u\over\partial n}\right| <q  \Rightarrow  u=0,\mbox{   on  }
\Gamma_{2}\times(0, T), \\
&&\left|{\partial u\over\partial n}\right| =q  \Rightarrow \exists k>0 : \quad u=-
k{\partial
u\over\partial n},
\mbox{   on  }  \Gamma_{2}\times(0, T),
\end{eqnarray*}
 \begin{eqnarray*}
 u=b  \quad \mbox{on} \quad\Gamma_{1}\times(0, T),
\end{eqnarray*}
\begin{eqnarray*}
- {\partial u\over\partial n}= f  \quad \mbox{on} \quad\Gamma_{3}\times(0, T),
\end{eqnarray*}
with the initial condition
 \begin{eqnarray*}
  u(0)= u_{b} \quad \mbox{on} \quad\Omega,
 \end{eqnarray*}
and the compatibility condition on $\Gamma_{1}\times(0, T)$
 \begin{eqnarray*}
  u_{b} = b\quad \mbox{on} \quad\Gamma_{1}\times(0, T)
 \end{eqnarray*}
\end{problem}
where $q>0$ is the Tresca friction coefficient on  $\Gamma_{2}$
 (\cite{ACF2002}, \cite{MB-EM2008}, \cite{Duvaut-JJL72}). We define the spaces
$\mathcal{F}=L^{2}((0,T)\times \Gamma_{3})$,
 $V=H^{1}(\Omega)$,  $V_{0}=\{v\in V :
v_{|_{\Gamma_{1}}}=0\}$, $H= L^{2}(\Omega)$,
$\mathcal{H}=L^{2}(0,T; H)$, $\mathcal{V}=L^{2}(0,T; V)$
 and the closed convex set $K_{b}=\{v\in V :  v_{|_{\Gamma_{1}}}=b\}$.
Let given
\begin{eqnarray}\label{hyp}
g\in \mathcal{H},\quad  b\in L^{2}(0 , T, H^{1/2}(\Gamma_{1})),
\quad f \in \mathcal{F}\nonumber\\
q\in L^{2}((0 , T)\times \Gamma_{2}), \,\,  q>0,\quad  u_{b}
\in K_{b}.
 \end{eqnarray}

The variational formulation of {\it Problem} \ref{pb4.1}  leads to
the following parabolic variational problem:
\begin{problem}\label{pbf}
Let given $g$, $b$,  $q$, $u_{b}$ and  $f$ as in (\ref{hyp}).
 Find $u=u_{f}\in \mathcal{C}(0, T, H)\cap
L^{2}(0 , T; K_{b})$ with $\dot{u}\in \mathcal{H}$, such that
$u(0)=u_{b}$, and for $t\in (0 , T)$
\begin{eqnarray*}
 <\dot{u}  , v-u> + a(u , u-v) +\Phi(v) -\Phi(u)  \geq (g , v-u) \nonumber\\- \int_{\Gamma_{3}}f (v-u)ds,
\quad
\forall v\in K_{b}.\label{iv1f}
\end{eqnarray*}
where  $(\cdot , \cdot)$ is the scalar product in $H$,  $a$ and $\Phi$ are defined by
\begin{eqnarray}\label{phi}
a(u , v)= \int_{\Omega} \nabla u\nabla v dx, \mbox{ and }  \Phi(v)=\int_{ \Gamma_{2}}q|v| ds.
\end{eqnarray}
\end{problem}
The functional $\Phi$ comes from the Tresca condition on $\Gamma_{2}$ \cite{ACF2002},
\cite{MB-EM2008}.
We consider also the following problem where we  change, in {\it
Problem} \ref{pb4.1}, only the Dirichlet condition on
$\Gamma_{1}\times(0, T)$ by the Newton law or a Robin boundary
condition i.e.
\begin{problem}\label{pb4.1-h}
\begin{eqnarray*}
 \dot{u} - \Delta u = g  \quad \mbox{in} \quad \Omega\times(0, T), \label{III.1}
\end{eqnarray*}
\begin{eqnarray*}
&&\left|{\partial u\over\partial n}\right| <q  \Rightarrow  u=0,
\mbox{   on  }  \Gamma_{2}\times(0, T),\\
&&\left|{\partial u\over\partial n}\right| =q  \Rightarrow \exists k>0 : \quad u=-
k{\partial
u\over\partial n},
\mbox{   on  }  \Gamma_{2}\times(0, T),
\end{eqnarray*}
 \begin{eqnarray*}
- {\partial u\over\partial n}= h(u-b)   \mbox{ on } \Gamma_{1}\times(0,
T), \quad
\end{eqnarray*}
\begin{eqnarray*}
- {\partial u\over\partial n}= f  \mbox{ on } \Gamma_{3}\times(0,
T),
\end{eqnarray*}
with the initial condition
 \begin{eqnarray*}
  u(0)= u_{b} \quad \mbox{on} \quad\Omega,
 \end{eqnarray*}
and the condition of compatibility on $\Gamma_{1}\times(0, T)$
 \begin{eqnarray*}
  u_{b} = b\quad \mbox{on} \quad\Gamma_{1}\times(0, T).
 \end{eqnarray*}
\end{problem}
The variational formulation of the problem (\ref{pb4.1-h})  leads to
the following parabolic variational problem
\begin{problem}\label{pbf-h}
Let given $g$, $b$,  $q$, $u_{b}$ and  $f$ as in (\ref{hyp}). For all $h>0$, find $u=u_{hf}$ in
 $\mathcal{C}(0, T, H)\cap \mathcal{V}$ with $\dot{u}$ in $\mathcal{H}$, such
that $u(0)=u_{b}$, and for $t\in (0 , T)$
\begin{eqnarray*}
 <\dot{u}  , v-u> + a_{h}(u , u-v) +\Phi(v) -\Phi(u)  \geq (g , v-u)\nonumber\\
- \int_{\Gamma_{3}}f (v-u)ds
+ h \int_{\Gamma_{1}}b (v-u)ds, \quad \forall v\in V,\label{iv1fh}
\end{eqnarray*}
where $a_{h}$ is defined by
 \begin{eqnarray*}
   a_{h}(u,v)= a(u,v)+ h\int_{\Gamma_{1}}u v
 ds.
 \end{eqnarray*}
\end{problem}
Moreover from {\rm\cite{KS1980, R1987, TT, T} we have that:
 $ \exists\lambda_{1}>0$  such that  $$\lambda_{h}\|v\|_{V}^{2}
\leq a_{h}(v , v)  \quad \forall v\in V, \mbox{ with }  \lambda_{h}=
\lambda_{1}\min\{1  \,, \, h\}$$ that is,  $a_{h}$ is also a
bilinear, continuous, symmetric and coercive form  $V\times V$ to $\br$.
The existence and uniqueness of the solution to each of the above
{\it Problem} \ref{pbf}  and {\it Problem} \ref{pbf-h}, is well
known see for example \cite{Br1972}, \cite{Chipot2000},
\cite{Duvaut-JJL72}.

The main goal of this paper is to prove in Section \ref{sec2} the existence and uniqueness of a family of
{\it optimal control problems}  \ref{np} and \ref{nph}
where the control variable is given
by a boundary condition of Neumann type whose state system is governed by a free boundary problem with Tresca conditions on a portion
$\Gamma_{2}$ of the boundary,  with a flux $f$ on $\Gamma_{3}$ {\it as
the control variable}, using a regularization method to overcome the nondifferentiability of the functional $\Phi$.
 Then in Section  \ref{secLim}  we study the convergence when $h\to +\infty$ of  the state systems and optimal controls
associated to the problem \ref{nph} to the corresponding state system and optimal control
associated to problem \ref{np}.  In order to obtain this last result we obtain an auxiliary strong  convergence by using
the Aubin compactness arguments see Lemma \ref{lem1.4}. This paper completes our previous paper \cite{MB-DT3}
and solves the open problem left in \cite{MB-DT4}.

Remark here that our study still valid with the bilinear  form $a$ in more general cases,  provided that $a$
must be symmetric, coercive and continuous from $V\times V$ to $\br$.

\section{Boundary optimal control problems}\label{sec2}
Let $M>0$ be a constant and we define the space
$$\mathcal{F_{-}}=\{f\in \mathcal{F} : \quad f\leq 0\}.$$
We consider the following Neumannn boundary optimal control problems defined by
 \cite{KeMu2008, KeSa1999,  JLL, T2010}
\begin{problem}\label{np}
 Find the optimal control $f_{op}\in \mathcal{F_{-}}$ such that
\begin{eqnarray}\label{npeq1}
 J(f_{op}) = \min_{f\in \mathcal{F_{-}}} J(f)
\end{eqnarray}
\end{problem}
where the cost functional $J: \mathcal{F_{-}} \to \br^{+}$ is given
by
\begin{eqnarray}\label{npeq2}
J(f)= {1\over 2}\|u_{f}\|^{2}_{\mathcal{H}} + {M\over 2}\|f\|^{2}_{\mathcal{F}}
\,\, (M>0)
\end{eqnarray}
and  $u_{f}$ is the unique  solution of the  {\it Problem} \ref{pbf} for a given $f\in \mathcal{F_{-}}$.
\begin{problem}\label{nph}
Find the optimal control $f_{{op}_{h}}\in \mathcal{F_{-}}$ such that
\begin{eqnarray}
 J(f_{{op}_{h}}) = \min_{f\in \mathcal{F_{-}}} J_{h}(f)
\end{eqnarray}
\end{problem}
where the cost functional $J_{h}: \mathcal{F_{-}} \to \br^{+}$ is
given by
\begin{eqnarray}
 J_{h}(f)= {1\over 2}\|u_{hf}\|^{2}_{\mathcal{H}} + {M\over 2}\|f\|^{2}_{\mathcal{F}}
\,\, (M>0,  \, h>0)
\end{eqnarray}
and  $u_{hf}$ is the unique solution of {\it Problem}  \ref{pbf-h} for a given $f\in \mathcal{F_{-}}$ and $h>0$.

\begin{theorem}\label{th3}
 Under the assumptions $g\geq 0$ in $\Omega\times (0 , T)$,
$b\geq 0$ on $\Gamma_{1}\times (0 , T)$ and  $u_{b}\geq 0$ in $\Omega$,
we have the following properties:\newline
\noindent(a) The cost functional  $J$ is strictly convex on $\mathcal{F_{-}}$,
 \newline
\noindent(b) There exists a unique optimal control $f_{op}\in \mathcal{F_{-}}$
solution of the Neumann boundary optimal control Problem \ref{np}.
\end{theorem}
\begin{proof}
We give some sketch of the proof, following \cite {MB-DT3} we generalize for parabolic variational
inequalities of the second kind, given in {\it Problem} \ref{pbf}, the estimates obtained for convex combination
  between $u_{4}(\mu)= u_{\mu f_{1}+(1-\mu) f_{2}}$,
and  $u_{3}(\mu)= \mu u_{f_{1}}+(1-\mu) u_{f_{2}}$, for any two
element $f_{1}$ and $f_{2}$ in $\mathcal{F}$. The main difficulty, to prove this result comes from the fact
that the functional $\Phi$ is not differentiable. To overcome this
difficulty, we use the regularization  method and consider for
$\varepsilon >0$ the following approach of $\Phi$ defined by

\begin{eqnarray}
\Phi_{\varepsilon}(v)= \int_{\Gamma_{2}}q\sqrt{\varepsilon^{2} +
|v|^{2}} ds, \qquad \forall v\in V,
\end{eqnarray}
which is Gateaux differentiable, with
$$\langle \Phi'_{\varepsilon}(w) \,, \, v\rangle =  \int_{\Gamma_{2}}
{qw v \over \sqrt{\varepsilon^{2} +|w|^{2}}} ds \qquad \forall (w ,
v)\in V^{2}.
$$
We define $u^{\varepsilon}$ as the unique solution of the
corresponding parabolic variational inequality for all $\varepsilon
>0$. We obtain that for all  $\mu\in [0 , 1]$ we have
$u^{\varepsilon}_{4}(\mu)\leq u^{\varepsilon}_{3}(\mu)$ for all
$\varepsilon >0$.

 When $\varepsilon\to 0$ we have that: for
 $i= 1,\cdots, 4$.
\begin{eqnarray}
u^{\varepsilon}_{i}\to u_{i}\mbox{ strongly in } \mathcal{V}\cap L^{\infty}(0 , T; H).
\end{eqnarray}
As  $f\in \mathcal{F_{-}}$, $g\geq 0$ in $\Omega\times (0 , T)$,
$b\geq 0$ in $\Gamma_{1}\times (0 , T)$ and $u_{b}\geq 0$ in $\Omega$, we obtain by the weak maximum principle that
for all  $\mu\in [0 , 1]$ we have  $0 \leq u_{4}(\mu)$, so following \cite{MB-DT3} we get
\begin{equation}
 0 \leq u_{4}(\mu) \leq u_{3}(\mu)\quad in \quad \Omega\times[0 , T],
    \quad \forall \mu\in [0 , 1].
\end{equation}
Then for all $\mu\in ]0 , 1[$, and  for all $f_{1}, f_{2}$ in
$\mathcal{F_{-}}$, and by using $f_{3}(\mu)= \mu f_{1} + (1-\mu)f_{2}$
we obtain that:
\begin{eqnarray}
&&\mu J(f_{1})+ (1-\mu)J(f_{2})- J(f_{3}(\mu))=
\nonumber\\
&& {1 \over
2}\left(\|u_{3}(\mu)\|_{\mathcal{H}}^{2}
-\|u_{4}(\mu)\|_{\mathcal{H}}^{2}\right)
+{1 \over 2}\mu(1-\mu) \|u_{f_{1}}-u_{f_{2}}\|_{\mathcal{H}}^{2}
\nonumber\\
&&+{M
\over 2}\mu(1-\mu) \|f_{1}-f_{2}\|_{\mathcal{F}}^{2}.\qquad
\end{eqnarray}
Then $J$ is strictly convex functional on $\mathcal{F_{-}}$ and
therefore there exists a unique optimal $f_{op}\in \mathcal{F_{-}}$
solution of the Neumann boundary optimal control {\it Problem }\ref{np}. $\Box$
\end{proof}

\begin{theorem}
 Under the assumptions $g\geq 0$ in $\Omega\times (0 , T)$,
$b\geq 0$ in $\Gamma_{1}\times (0 , T)$ and $u_{b}\geq 0$ in $\Omega$,
 we have the following properties:\newline
\noindent(a) The cost functional  $J_{h}$ are strictly convex on
 $\mathcal{F_{-}}$, for all $h>0$,\newline
\noindent(b) There exists a unique optimal control $f_{h_{op}}\in \mathcal{F_{-}}$
solution of the Neumann boundary optimal control Problem
\ref{nph}, for all $h>0$.
\end{theorem}
\begin{proof}
We follow a similar method to the one developed in {Theorem
\ref{th3}} for all $h>0$.  $\Box$
\end{proof}

\section{Convergence when $h\to +\infty$}\label{secLim}

In this  section we study the  convergence of the Neumann optimal control {\it
Problem} \ref{nph} to the optimal control {\it
Problem} \ref{np}  when $h\to \infty$. For a given $f\in \mathcal{F}$   we have first the
following result which generalizes \cite{MB-DT2, MB-DT3, TT, T}.

\begin{lemma}\label{l6.1}  Let $u_{h f}$ be the unique solution of the  problem  {\rm\ref{pbf-h}}
and $u_{f}$ the unique solution of the problem {\rm\ref{pbf}}, then
 $$u_{h f}\to u_{f}\in \mathcal{V}  \mbox{ strongly  as } h\to +\infty, \qquad \forall f\in \mathcal{F}.$$
\end{lemma}
\begin{proof} Following \cite{MB-DT3},
we take  $v=u_{f}(t)$ in the variational inequality of the problem {\rm\ref{pbf-h}}
 where $u= u_{h  f}$,
and recalling that $u_{f}(t)= b$ on $\Gamma_{1}\times]0 , T[$,
 taking $\phi_{h}(t)= u_{h f}(t)- u_{f}(t)$   we obtain for $h>1$, 
that
$\|u_{h f}\|_{\mathcal{V}}$
is also bounded for all $h>1$ and for all $f\in \mathcal{F}$.
Then there exists $\eta\in \mathcal{V}$ such that (when $h\to +\infty$)
\begin{eqnarray*}
u_{h f}\tow \eta \mbox{ weakly in } \mathcal{V}
\end{eqnarray*}
and
\begin{eqnarray*}
u_{h f}\to b  \quad\mbox{strongly on  } L^{2}((0 , T)\times\Gamma_{1})
\end{eqnarray*}
 so $\eta(0)= u_b$.

Let $\varphi$ be in $L^{2}(0 , T , H^{1}_{0}(\Omega))$ and taking in
the variational inequality of the problem {\rm\ref{pbf-h}} where
$u= u_{h f}$,
$v= u_{h f}(t) \pm \varphi(t)$, we obtain
as $\|u_{h f}\|_{\mathcal{V}}$
is bounded for all $h>1$, we deduce that
$\|\dot{u}_{h f}\|_{L^{2}(0 , T , H^{-1}(\Omega))}$ is also bounded for all $h>1$.
Then we conclude that
 \begin{eqnarray}\label{eqW2}
\left.
\begin{array}{ll}
u_{h f}  \tow \eta  \mbox{ in }\mathcal{V} \mbox{ weak,  }
 \mbox{ and in } L^{\infty}(0 , T, H)
\mbox{ weak star,} \\
\mbox{and }
 \dot{u}_{h f}\tow \dot{\eta}\mbox{ in }L^{2}(0 , T, H^{-1}(\Omega)) \mbox{ weak}.
\end{array}
\right\}\quad
\end{eqnarray}

From the variational inequality of the problem {\rm\ref{pbf-h}}
 and taking $v\in K$ so $v= b$ on $\Gamma_{1}$, we obtain $a.e. \, t\in ]0 , T[$
\begin{eqnarray*}
\langle \dot{u}_{h f}, v- u_{h f}\rangle +
a(u_{h f} , v- u_{h f}) -  h\int_{\Gamma_{1}} |u_{h f}- b|^{2}ds  \geq
\nonumber\\
 \Phi(u_{h f}) - \Phi(v) + ( g , v - u_{h f})
- \int_{\Gamma_{3}} f (v - u_{h f}) ds, \qquad
\end{eqnarray*}
for all $v\in K$,
then as $h>0$ we have  a.e.  $t\in ]0 , T[$.
\begin{eqnarray}\label{eq6.1}
\langle \dot{u}_{h f}, v- u_{h f}\rangle +
a(u_{h f} , v- u_{h f})  \geq
  \Phi(u_{h f}) - \Phi(v) +\nonumber\\ (g , v - u_{h f})
- \int_{\Gamma_{3}} f (v - u_{h f}) ds,
 \,\, \forall v\in K. \qquad
\end{eqnarray}
So using (\ref{eqW2}) and passing to the limit when $h \to +\infty$ we obtain
\begin{eqnarray*}
\langle \dot{\eta}, v- \eta\rangle  +
a(\eta , v- \eta)   +\Phi(v)- \Phi(\eta) \geq ( g  , v - \eta)
\nonumber\\- \int_{\Gamma_{3}} f (v - \eta ) ds,
 \quad \forall v\in K  \quad a.e. \, t\in ]0 , T[,
\end{eqnarray*}
and $\eta(0)=u_b$.
Using the uniqueness of the solution of Problem \ref{pbf} we get that $\eta= u_{f}$.

To prove the strong convergence, we  take $v= u_{f}(t)$ in the variational inequality  of the {\it problem} \ref{pbf-h}
\begin{eqnarray*}
&&\langle \dot{u}_{h f}  , u_{f} - u_{h f} \rangle  +
a_{h}(u_{h f}  , u_{f} - u_{h f} ) +\Phi(u_{f})
\nonumber\\
&&- \Phi(u_{h f}) \geq ( g  ,  u_{f} - u_{h f} )
 + h\int_{\Gamma_{1}} b (u_{f}  - u_{h f} )ds
\nonumber\\
&&
- \int_{\Gamma_{3}} f (u_{f}  - u_{h f} )ds,
\end{eqnarray*}
a.e. \ $t\in ]0 , T[$, thus as $u_{f}= u_b$ on $\Gamma_{1}\times]0 , T[$,
we put \newline $\phi_{h}= u_{h f} - u_{f}$, so   a.e. $t\in ]0 , T[$ we have
\begin{eqnarray*}
\langle \dot{\phi_{h}} \, ,\,  \phi_{h}\rangle  +
a(\phi_{h} , \phi_{h})
+ h \int_{\Gamma_{1}} |\phi_{h}|^{2}ds + \Phi(u_{h f})-\Phi(u_{f})
\nonumber\\
\leq
\langle \dot{u}_{f} , \phi_{h}\rangle
 +
 a( u_{f} \, ,\, \phi_{h})  + ( g ,  \phi_{h})
- \int_{\Gamma_{3}} f \phi_{h} ds,\qquad
\end{eqnarray*}
so
\begin{eqnarray*}
{1\over 2}\|\phi_{h}\|_{L^{\infty}(0 , T, H)}^{2}  +
\lambda_{h}\|\phi_{h}\|_{\mathcal{V}}^{2}
 + \Phi(u_{h f})-\Phi(u_{f})\nonumber\\ \leq
-\int_{0}^{T}\langle \dot{u}_{f}(t), \phi_{h}(t)\rangle dt
-\int_{0}^{T} a(u_{f}(t), \phi_{h}(t))dt
\nonumber\\
+ \int_{0}^{T}( g(t)  , \phi_{h}(t)) dt
- \int_{0}^{T}\int_{\Gamma_{3}} f \phi_{h} ds dt.
\end{eqnarray*}
Using the weak semi-continuity of $\Phi$ and the weak convergence (\ref{eqW2})  the right side
of the just above inequality tends to zero when $h\to +\infty$, then we deduce the strong convergence
of $\phi_{h}=u_{h f}-u_{f}$ to $0$ in $\mathcal{V} \cap L^{\infty}(0 , T, H) $, for all
$f\in   \mathcal{F_{-}}$ and
 the proof holds.   $\Box$
\end{proof}

 We prove now the  following lemma by using the Aubin compactness arguments.
This Lemma \ref{lem1.4} is very important and necessary  which  allow us to conclude this paper.
Indeed this result is needed   to pass to the limit exactly  in the last term of
 the inequality (\ref{eqab})  in the proof of the main Theorem \ref{th6.1}.

\begin{lemma}\label{lem1.4}
Let $u_{{h f_{op}}_{h}}$ the state system defined by the unique solution of Problem {\rm\ref{pbf-h}}, where
the flux $f$ is replaced by $f_{op_{h}}$. Then,  for $h\to +\infty$, we have
\begin{eqnarray}\label{a}
 u_{{h f_{op}}_{h}}\to  u_{f}  \quad in \quad  L^{2}((0, T)\times\partial\Omega),
\end{eqnarray}
where $u_{f}$ is the the state system defined by the unique solution of Problem {\rm\ref{pbf}} with
the flux $f$ on $\Gamma_{3}$.
\end{lemma}
\begin{proof}
Let consider the variational inequality  of Problem {\rm\ref{pbf-h}} with $u=u_{{h f_{op}}_{h}}$
and $f=f_{op_{h}}$ i.e.
\begin{eqnarray}\label{aa}
 <\dot{u}_{{h f_{op}}_{h}}  , v-u_{{h f_{op}}_{h}}> + a_{h}(u_{{h f_{op}}_{h}} , v-u_{{h f_{op}}_{h}})
 +\Phi(v)
\nonumber\\
 -\Phi(u_{{h f_{op}}_{h}})  \geq (g , v-u_{{h f_{op}}_{h}})
- \int_{\Gamma_{3}}f_{op_{h}}      (v-u_{{h f_{op}}_{h}})ds
\nonumber\\
+ h \int_{\Gamma_{1}}b (v-u_{{h f_{op}}_{h}})ds, \quad \forall v\in V,\qquad
\end{eqnarray}
and let $\varphi \in L^{2}(0, T; H^{1}_{0}(\Omega))$, and set $v = u_{{h f_{op}}_{h}}(t) \pm \varphi(t)$ in (\ref{aa}),
we get
\begin{eqnarray*}
 <\dot{u}_{{h f_{op}}_{h}}  , \varphi>
   = (g , \varphi) - a(u_{{h f_{op}}_{h}} , \varphi).
\end{eqnarray*}
By integration in times for $t\in (0 , T)$, we get
\begin{eqnarray*}
\int_{0}^{T} <\dot{u}_{{h f_{op}}_{h}}  , \varphi> dt
   = \int_{0}^{T}(g , \varphi)dt - \int_{0}^{T}a(u_{{h f_{op}}_{h}} ,  \varphi)dt
\end{eqnarray*}
thus for $A=(c {\|g\|}_{\mathcal{H}} +  {\|u_{{h f_{op}}_{h}}\|}_{\mathcal{V}})$, we get
\begin{eqnarray*}
|\int_{0}^{T} <\dot{u}_{{h f_{op}}_{h}}  , \varphi> dt|\leq A
 {\|\varphi\|}_{L^{2}(0, T; H^{1}_{0}(\Omega))}
\end{eqnarray*}
 where $c$ comes from the Poincar\'e inequality,  and  as in Lemma \ref{l6.1} we can obtain that
$u_{{h f_{op}}_{h}}$ is bounded in $\mathcal{V}$, so there exists a  positive constant
$C$   such that
\begin{eqnarray}\label{b}
\|\dot{u}_{{h f_{op}}_{h}}\|_{L^{2}(0, T; H^{-1}(\Omega))}\leq C.
\end{eqnarray}
Using now the Aubin compactness arguments, see for example \cite{Foias} with the three
Banach spaces
$V$, $H^{{2\over 3}}(\Omega)$ and $H^{-1}(\Omega)$, then
\begin{eqnarray*}
 u_{{h f_{op}}_{h}}\to  u_{f}  \quad L^{2}(0, T; H^{{2\over 3}}(\Omega)).
\end{eqnarray*}
As the trace operator $\gamma_{0}$ is continuous from $H^{{2\over 3}}(\Omega)$
to $L^{2}(\partial\Omega)$, then the result follows.  $\Box$
\end{proof}

We  give now, without need to use the notion of adjoint states \cite{GT2008, JLL}, the convergence
 result which  generalizes the result obtained
 in \cite{MT2007} for a  parabolic variational equalities (see also
\cite{arada2000, BEM2003,   belgacem2003, GT2003, GT2008}). Other optimal control problems gouverned by variational
inequalities are given in \cite{Barbu1984, Reyes, Mingot1}.

\begin{theorem}\label{th6.1}
Let $u_{{h f_{op}}_{h}}\in \mathcal{V}$, ${f_{op}}_{h}\in \mathcal{F_{-}}$  and $u_{f_{op}}\in \mathcal{V}$,
$f_{op}\in \mathcal{F_{-}}$ be respectively
 the state systems and the optimal controls defined in the problems {\rm(\ref{pbf-h})}
 and {\rm(\ref{pbf})}.
Then
\begin{eqnarray}\label{6.1}
 \lim_{h\to +\infty}\|u_{h f_{op_{h}}}-u_{f_{op}}\|_{\mathcal{V}}=\nonumber\\
=
 \lim_{h\to +\infty}\|u_{h f_{op_{h}}}-u_{f_{op}}\|_{L^{\infty}(0 , T , H)} ,
\nonumber\\
= \lim_{h\to +\infty}\|u_{h f_{op_{h}}}-u_{f_{op}}\|_{L^{2}((0 , T)\times\Gamma_{1})}= 0,\qquad
\end{eqnarray}

\begin{equation}\label{6.2}
 \lim_{h\to +\infty}\|f_{op_{h}}-f_{op}\|_{\mathcal{F}}= 0.
\end{equation}
\end{theorem}
\begin{proof}
We have first
\begin{eqnarray*}
 J_{h}(f_{op_{h}})= {1\over 2}\|u_{h f_{op_{h}}}\|_{\mathcal{H}}^{2}
 + {M\over 2}\|f_{op_{h}}\|_{\mathcal{F} }^{2}\leq
\nonumber\\
\leq {1\over 2}\|u_{h f}\|_{\mathcal{H}}^{2} + {M\over 2}\|f\|_{\mathcal{F} }^{2},
\end{eqnarray*}
for all $f\in \mathcal{F}_{-}$,
then for $f=0\in \mathcal{F}_{-} $ we obtain that
\begin{eqnarray}\label{e6.3}
J_{h}(f_{op_{h}})= {1\over 2}\|u_{h f_{op_{h}}}\|_{\mathcal{H}}^{2}
+ {M\over 2}\|f_{op_{h}}\|_{\mathcal{F}}^{2}\leq {1\over 2}\|u_{h 0}\|_{\mathcal{H}}^{2}\qquad
\end{eqnarray}
where $u_{h 0}\in \mathcal{V}$  is the solution of the following parabolic variational inequality
\begin{eqnarray*}
\langle \dot{u}_{h 0} , v- u_{h 0}\rangle +
 a_{h}( u_{h 0} , v- u_{h 0}) +\Phi(v)- \Phi(u_{h 0})
 \nonumber\\ \geq
 \int_{\Omega} g(v- u_{h 0})dx
+     h\int_{\Gamma_{1}} b(v- u_{h 0})ds, \quad a.e. \, t\in ]0 , T[
\end{eqnarray*}
for all $v\in V$ and $u_{h 0}(0)= u_{b}$.

 Taking $v= u_{b}\in K_b $  we get that
   $\|u_{h 0}- u_{b}\|_{\mathcal{V}}$ is bounded independently of $h$, then
  $\|u_{h 0}\|_{\mathcal{H}}$ is bounded independently of $h$. So we deduce with (\ref{e6.3}) that
 $\|u_{h f_{op_{h}}}\|_{\mathcal{H}}$ and $\|f_{op_{h}}\|_{\mathcal{F} }$ are also bounded independently  of $h$.
So there exist $\tilde{f}\in \mathcal{F}_{-}$ and $\eta$ in  $\mathcal{H}$ such that
\begin{eqnarray}\label{6.5}
 f_{op_{h}} \tow  \tilde{f} \mbox{ in }  \mathcal{F}_{-}  \,\, {\rm and }\,\,  u_{h f_{op_{h}}}\tow\eta
\mbox{ in } \mathcal{H} \,\,  (weakly).
\end{eqnarray}

Taking  now $v=u_{f_{op}}(t)\in K_b$ in {\it Problem } (\ref{pbf-h}), for $t\in ]0 , T[$,  with $u= u_{h f_{op_{h}}}$
and $f=f_{op_{h}}$, we obtain
\begin{eqnarray*}
\langle \dot{u}_{h f_{op_{h}}} , u_{f_{op}} - u_{h f_{op_{h}}}\rangle +
 a_{1}(  u_{h f_{op_{h}}} , u_{f_{op}} -  u_{h f_{op_{h}}})\qquad
\nonumber\\
+(h-1) \int_{\Gamma_{1}}   u_{h f_{op_{h}}}   (u_{f_{op}}  -  u_{h f_{op_{h}}} )ds
+ \Phi(u_{f_{op}})\qquad
\nonumber\\
- \Phi( u_{h f_{op_{h}}})
\geq
 ( g  , u_{f_{op}}  -  u_{h f_{op_{h}}} )
+
h\int_{\Gamma_{1}} b (u_{f_{op}}  - u_{h f_{op_{h}}} )ds
\nonumber\\
-\int_{\Gamma_{3}}    f_{op_{h}}  (u_{f_{op}}  -  u_{h f_{op_{h}}} )ds
, \quad a.e. \, t\in ]0 , T[.\qquad
\end{eqnarray*}
As $u_{f_{op}}= b$ on $\Gamma_{1}\times[0 , T]$, taking
$\phi_{h}= u_{f_{op}} - u_{h f_{op_{h}}}$
 we obtain
\begin{eqnarray*}
{1\over 2}\|\phi_{h}\|_{L^{\infty}(0 , T;  H)}^{2}
+ \lambda_{1}\|\phi_{h}\|_{\mathcal{V}}^{2}
+(h-1) \int_{0}^{T}\int_{\Gamma_{1}}|\phi_{h}(t)|^{2}ds dt\qquad
\nonumber\\ \leq
 \int_{0}^{T}\int_{\Gamma_{3}}f_{op_{h}}  \phi_{h} ds dt
-\int_{0}^{T}(g(t)  , \phi_{h}(t)) dt\qquad\qquad
\nonumber\\
+ \int_{0}^{T}\int_{\Gamma_{2}} q  |\phi_{h}(t)|dsdt
+\int_{0}^{T}\langle \dot{u}_{f_{op}}(t) \phi_{h}(t)\rangle dt\qquad
\nonumber\\\qquad
+
\int_{0}^{T}a( u_{f_{op}}(t) , \phi_{h}(t))dt.\qquad\qquad
\end{eqnarray*}
As   $f_{op_{h}}$ is bounded in    $\mathcal{F}_{-}$,  from (\ref{b})
$\dot{u}_{f_{op}}$  is bounded  in $L^{2}(0, T; H^{-1}(\Omega))$,
and $u_{h f_{op_{h}}}$ is also bounded in $\mathcal{V}$,
all independently on $h$, so   there exists a positive constant $C$
which does not depend on $h$ such that
\begin{eqnarray*}\label{6.q}
\|\phi_{h}\|_{\mathcal{V}}= \|u_{h f_{op_{h}}} -u_{f_{op}}\|_{\mathcal{V}}\leq C,
\quad \|\phi_{h}\|_{L^{\infty}(0 , T, H)}\leq C
\nonumber\\
\mbox{ and  }
 (h-1)\int_{0}^{T}\int_{\Gamma_{1}} |u_{h f_{op_{h}}}- b|^{2}ds dt \leq C,
\end{eqnarray*}
then $\eta \in \mathcal{V}$ and
\begin{eqnarray}\label{6.6}
u_{h f_{op_{h}}} \tow \eta \mbox{ in } \mathcal{V} \, \mbox{ and in }  L^{\infty}(0 , T ,  H)
 \mbox{ weak star }
\end{eqnarray}
\begin{eqnarray}\label{6.7}
 u_{h f_{op_{h}}} \to  b\quad in \quad L^{2}((0 , T)\times\Gamma_{1}) \,  \mbox{ strong},
\end{eqnarray}
so $\eta(t)\in K_b$ for all $t\in [0 , T]$.

Now taking $v\in K$  in {\it Problem }(\ref{pbf-h}) where $u= u_{h f_{op_h}}$ and $f=f_{op_{h}}$ so
 \begin{eqnarray*}\label{}
 \langle \dot{u}_{h f_{op_{h}}}  , v- u_{h f_{op_{h}}}  \rangle +
 a_{h}(u_{h f_{op_{h}}}         , v- u_{h f_{op_{h}}}   )
+\Phi(v)
\nonumber\\
- \Phi(u_{h f_{op_{h}}} )
\geq ( f_{op_{h}}  ,  v- u_{h f_{op_{h}}})
+ h\int_{\Gamma_{1}} b(v- u_{h f_{op_{h}}}  )ds
\nonumber\\
- \int_{\Gamma_{3}} f_{op_{h} }    (v- u_{h f_{op_{h}}}  )ds
, \quad a.e. \, t\in ]0 , T[
 \end{eqnarray*}
as $v\in K_b$ so $v=b$ on $\Gamma_{1}$, thus we have
 \begin{eqnarray*}\label{}
&& \langle \dot{u}_{h f_{op_{h}}}  , u_{h f_{op_{h}}} -v  \rangle
+ a(u_{h f_{op_{h}}}  ,    u_{h f_{op_{h}}}  -v) +
\nonumber\\
&&
 h \int_{\Gamma_{1}}|u_{h f_{op_{h}}} -b|^{2}ds  + \Phi(u_{h f_{op_{h}}}  )
-\Phi(v) -( g ,  v- u_{h f_{op_{h}}} )
\nonumber\\
&&\leq\int_{\Gamma_{3}} f_{op_{h} }    (v- u_{h f_{op_{h}}}  )ds
 \quad a.e. \, t\in ]0 , T[.
 \end{eqnarray*}

Thus
\begin{eqnarray}\label{eqab}
\langle \dot{u}_{h f_{op_{h}}}  ,   u_{h f_{op_{h}}} -v \rangle + a( u_{h f_{op_{h}}}   , u_{h f_{op_{h}} } -v)
\nonumber\\
 + \Phi(u_{h f_{op_{h}}} ) -\Phi(v)
\leq
-( g , v- u_{h f_{op_{h}}} )
\nonumber\\
 -     \int_{\Gamma_{3}} f_{op_{h} }    (v- u_{h f_{op_{h}}}  )ds
 \quad a.e. \, t\in ]0 , T[.
 \end{eqnarray}

Using Lemma \ref{lem1.4},  (\ref{6.5}) and (\ref{6.6}),  we deduce that \cite{Duvaut-JJL72, Ta1982}
\begin{eqnarray*}\label{}
 \langle \dot{\eta} , v - \eta \rangle +  a(\eta , v- \eta)
 +\Phi(v)-  \Phi(\eta)
 \geq  (f ,  v -\eta)
\nonumber\\
 -  \int_{\Gamma_{3}} \tilde{f}    (v- \eta) )ds       , \quad \forall v\in K, \quad a.e.  \, t\in]0 , T[,
 \end{eqnarray*}
so also by the uniqueness of the solution of {\it Problem }(\ref{pbf}) we obtain that
\begin{eqnarray}\label{xi}
u_{\tilde{f}}  = \eta.
 \end{eqnarray}

We prove that  $\tilde{f} = f_{op}$.  Indeed we have
 \begin{eqnarray*}\label{dd}
J(\tilde{f} )&=&{1\over 2} \|\eta \|_{\mathcal{H}}^{2} + {M\over 2} \|\tilde{f} \|_{\mathcal{F}}^{2}
\nonumber\\
 &\leq&\liminf_{h\to +\infty} \left\{{1\over 2} \|u_{h f_{op_{h}}}\|_{\mathcal{H}}^{2}
 + {M\over 2} \|f_{op_{h}}\|_{   \mathcal{F} }^{2} \right\}
\nonumber\\
&=&\liminf_{h\to +\infty} J_{h}(f_{op_{h}})
\nonumber\\
&\leq& \liminf_{h\to +\infty} J_{h}(f)
=\liminf_{h\to +\infty}
\left\{{1\over 2} \|u_{h f}\|_{\mathcal{H}}^{2} + {M\over 2} \|f\|_{\mathcal{F}}^{2} \right\}
 \end{eqnarray*}
so using now the strong convergence $u_{h f}\to u_{f}$ as\newline $h\to
+\infty,\; \forall \; f\in \mathcal{F}_{-} $ (see Lemma \ref{l6.1}), we obtain that
 \begin{eqnarray}\label{6.9}
J(\tilde{f} )\leq \liminf_{h\to +\infty} J_{h}(f_{op_{h}}) &\leq&
 {1\over 2} \|u_{f}\|_{\mathcal{H}}^{2} + {M\over 2} \|f\|_{\mathcal{F}}^{2}
\nonumber\\
&=& J(f),
\quad \forall f\in \mathcal{F}_{-}
 \end{eqnarray}

then by the uniqueness of the optimal control {\it Problem} (\ref{pbf}) we get
 \begin{eqnarray}\label{f}
\tilde{f} = f_{op}.
 \end{eqnarray}

Now we prove the  strong convergence of  $u_{h f_{op_{h}}}$ to $\eta=u_{f}$ in
$\mathcal{V}\cap L^{\infty}(0 , T ; H)\cap L^{2}(0 , T ; L^{2}(\Gamma_{1}))$,
 indeed taking $v=\eta$ in {\it Problem} (\ref{pbf-h}) where $u=u_{h f_{op_{h}}}$ and $f= f_{op_{h}}$,
as $\eta(t)\in K$ for $t\in [0 , T]$,  so $\eta =b$ on $\Gamma_{1}$, we obtain
\begin{eqnarray*}
&&{1\over 2} \|u_{h f_{op_{h}}} -\eta\|_{L^{\infty}(0 , T; H)}^{2} +
 \lambda_{1}\|u_{h f_{op_{h}}}-\eta\|_{\mathcal{V}}^{2} +
\nonumber\\
&&\int_{0}^{T} \{\Phi(u_{h f_{op_{h}}})- \Phi(\eta)\} dt
+\tilde{h} \|u_{h f_{op_{h}}} - \eta\|_{L^{2}((0 , T)\times\Gamma_{1})}^{2}
\nonumber\\
&&\leq \int_{0}^{T}( g , u_{h f_{op_{h}}}(t)-\eta(t))dt
- \int_{0}^{T}\langle \dot{\eta} ,   u_{h f_{op_{h}}}-\eta   \rangle dt +
\nonumber\\
&&
\int_{0}^{T}a(\eta(t) , \eta(t) - u_{h f_{op_{h}}}(t))
 - \int_{\Gamma_{3}} f_{op_{h}} (u_{h f_{op_{h}}}- \eta) )ds dt.
 \end{eqnarray*}
where $\tilde{h}=h-1$

Using (\ref{6.6}) and the weak semi-continuity of $\Phi$ we deduce that
\begin{eqnarray*}\label{}
\lim_{h\to +\infty}\|u_{h f_{op_{h}}} -\eta\|_{L^{\infty}(0 , T; H)}
=\lim_{h\to +\infty}\|u_{h f_{op_{h}}} -\eta\|_{\mathcal{V}} \nonumber\\
= \|u_{h f_{op_{h}}} - \eta\|_{L^{2}((0 , T)\times\Gamma_{1})} = 0,
 \end{eqnarray*}
and with (\ref{xi}) and (\ref{f}) we deduce (\ref{6.1}). Then from  (\ref{6.9}) and (\ref{f}) we can write
\begin{eqnarray}\label{eq6.11}
J(f_{op})&=&{1\over 2}\|u_{f_{op}}\|_{\mathcal{H}}^{2}
  + {M\over 2}\|f_{op}\|_{\mathcal{F} }^{2}\leq
\leq \liminf_{h\to+\infty} J_{h}(f_{op_{h}})\qquad \nonumber\\
&=& \liminf_{h\to+\infty} \left\{{1\over 2}\|u_{h f_{op_{h}}}\|_{\mathcal{H}}^{2}
+ {M\over 2}\|f_{op_{h}}\|_{\mathcal{F} }^{2}\right\}\nonumber\\
&\leq&  \lim_{h\to+\infty} J_{h}(f_{op}) = J(f_{op})
\end{eqnarray}
and using  the strong convergence (\ref{6.1}), we get
\begin{eqnarray}\label{eq6.13}
\lim_{h\to+\infty}\|f_{op_{h}}\|_{ \mathcal{F}}= \|f_{op}\|_{\mathcal{F}}.
\end{eqnarray}
Finally as
\begin{eqnarray}\label{eq6.14}
\|f_{op_{h}}- f_{op}\|_{\mathcal{F}}^{2} =
 \|f_{op_{h}}\|_{\mathcal{F}}^{2}+ \|f_{op}\|_{\mathcal{F}}^{2}
-2(f_{op_{h}} , f_{op})\nonumber\\
\end{eqnarray}
and by  the first part of (\ref{6.5}) we have
$$\lim_{h\to +\infty}\left(f_{op_{h}} , f_{op}\right) = \|f_{op}\|_{\mathcal{F}}^{2},$$
so from (\ref{eq6.13}) and (\ref{eq6.14}) we  get \mbox{(\ref{6.2})}. This ends  the proof. $\Box$
\end{proof}

\bigskip
\begin{corollary}
 Let $u_{{h f_{op}}_{h}}$ in $\mathcal{V}$, ${f_{op}}_{h}$ in $\mathcal{F_{-}}$,
  $u_{f_{op}}$ in $\mathcal{V}$ and
$f_{op}$ in $\mathcal{F_{-}}$ be respectively
 the state systems and the optimal controls defined in the problems {\rm(\ref{pbf-h})}
 and {\rm(\ref{pbf})}.
Then
\begin{eqnarray*}
 \lim_{h\to +\infty}|J_{h}(f_{op_{h}})- J({f_{op}})|= 0.
\end{eqnarray*}
\end{corollary}
\begin{proof}
It follows from the definitions (\ref{npeq1}) and  (\ref{npeq2}),  and the convergences (\ref{6.1}) and (\ref{6.2}).
 $\Box$
\end{proof}

\section{Conclusion}\label{Conclusion}
The main difference here with our work \cite{MB-DT3} where the control variable was the function $g$, is
that we consider here as a  control variable the function $f$ given by the Neumann boundary condition on $\Gamma_{3}$.
This change induce in the variational problems \ref{pbf}   and \ref{pbf-h}, and also in the proofs of
Lemma \ref{l6.1} and Theorem \ref{th6.1},  a new integral term on  $\Gamma_{3}$.
The main difficulty here is in Section \ref{secLim} and the question is exactly  how to pass to the limit for $h\to +\infty$
in the last integral term on  $\Gamma_{3}$ in (\ref{eqab}). To overcome this  main difficulty
we have introduced the new Lemma \ref{lem1.4}, which is the key of our problem.
The idea of Lemma \ref{l6.1} and Theorem \ref{th6.1}  and their proofs are indeed similar to those of our work
\cite{MB-DT3} with the differences and difficulties mentioned just above.

\begin{ack}
This paper was partially sponsored by the Institut Camille Jordan  ST-Etienne University for first author
and the project PICTO Austral $\#$ 73 from ANPCyT  and Grant AFOSR FA9550-10-1-0023 for the second author.
\end{ack}

}
\end{multicols}
\end{document}